\documentclass[12pt,reqno]{amsart}

\usepackage{amscd,amsmath}
\usepackage{amssymb,latexsym,amsthm}
\usepackage{amsfonts}
\usepackage{enumerate,color,graphicx}
\usepackage[all]{xypic}

\usepackage[pagebackref=false,colorlinks=true,linkcolor=red,citecolor=red]{hyperref}
\hypersetup{backref,pdfpagelabels=true,pdfpagemode=FullScreen,hypertex,hyperindex=true}

\makeatletter
\def\@seccntformat#1{%
  \protect\textup{%
    \protect\@secnumfont
    \expandafter\protect\csname format#1\endcsname 
    \csname the#1\endcsname
    \protect\@secnumpunct
  }%
}

\makeatother

\theoremstyle{theorem}
\newtheorem{theo}{Theorem}[section]
\newtheorem{coro}[theo]{Corollary}
\newtheorem{lemm}[theo]{Lemma}
\newtheorem{prop}[theo]{Proposition}

\newtheorem{theosub}{Theorem}[subsection]
\newtheorem{corosub}[theosub]{Corollary}
\newtheorem{lemmsub}[theosub]{Lemma}
\newtheorem{propsub}[theosub]{Proposition}

\theoremstyle{definition}

\newtheorem{rema}[theo]{Remark}
\newtheorem{remas}[theo]{Remarks}

\newtheorem{remasub}[theosub]{Remark}

\def\nid{\noindent}
\def\cf{\mathfrak{c}}
\def\of{\mathfrak{o}}
\def\smin{\smallsetminus}
\def\sue{\subseteq}
\def\OS{\text{\sf O}}

\def\SL{\mathsf{S}}
\def\TL{\mathsf{T}}
\newcommand{\ur}{{{\rlap{$\ $}\hbox{$\uparrow$}}}}%
\def\setof#1#2{\{#1 \ | \ #2\} }
\def\set#1{\{#1\} }
\def\setof#1#2{\{#1 \, |\, #2\} }
\def\qtq#1{\quad\text{#1}\quad}
\def\MC{{\sf M}}
\def\op{^{\text{\rm op}}}
\def\bim{\tbigwedge}
\def\TOP{\text{\bf Top}}
\def\frm{\text{\bf Frm}}
\def\loc{\text{\bf Loc}}
\newcommand{\tbigcap}{\mathop{\textstyle \bigcap}}
\newcommand{\tbigcup}{\mathop{\textstyle \bigcup }}
\newcommand{\tbigvee}{\mathop{\textstyle \bigvee }}
\newcommand{\tbigwedge}{\mathop{\textstyle \bigwedge }}
\def\intr{{\mathrm{int}}\,}
\def\clr{{\mathrm{cl}}\,}
\def\co{{^{\sf c}}}

\makeatletter
\@namedef{subjclassname@2020}{%
  \textup{2020} Mathematics Subject Classification}
\makeatother

\begin{document}

\title[Continuity and openness of maps on locales by way of adjunctions]{Continuity and openness of maps on locales\\[2mm] by way of Galois adjunctions}


\author[João Areias]{João Areias}

\address{\hspace*{-\parindent}Department of Mathematics, University of Coimbra,  3000-143 Coimbra, Portugal \newline {\it Email address}: {\tt uc2020217955@student.uc.pt}}

\author[Jorge Picado]{Jorge Picado}

\address{\hspace*{-\parindent}CMUC, Department of Mathematics, University of Coimbra,  3000-143 Coimbra, Portugal \newline {\it Email address}: {\tt picado@mat.uc.pt}}



\date{}

\begin{abstract}
We study four adjoint situations in pointfree topology that interchange images and preimages with closure and interior operators and establish with them a number of characterizations for meet-preserving maps, localic maps, open maps (in a broad sense) and open localic maps between locales.
The principal and most attractive feature of these adjunctions is that they are all
concerned with elementary ideas and basic concepts of localic topology: the use of the concrete language of sublocales and its technique simplifies the reasoning.

We then revisit open localic maps in detail and present a new proof of Joyal-Tierney
open mapping theorem.
We end with a study of the interchange laws between preimages/images and closure/interior operators,
making clear the similarities and differences with the classical realm.
\end{abstract}

\maketitle

\section{Introduction}

\bigskip
Given two partially ordered sets $X$ and $Y$, a {\em Galois connection} \cite{EKMS} between them consists of a pair of order-preserving maps $f \colon  X \to Y$ and $g \colon  Y \to X$ such that
$$
f(x) \leq y \iff x \leq g(y)
$$
for all $x\in X$ and $y\in Y$. Nowadays, one usually refers to this situation as a  {\em \textup(Galois\textup) adjunction}, since it is a special case of the key concept in category theory of an {\em adjunction} (in fact, a poset is simply a category in which there is at most one arrow between objects).  One calls $f$ {\em left adjoint} to $g$ and $g$ {\em right adjoint} to $f$ and writes $f\dashv g$.

 It is standard that (cf. \cite{EKMS,ME})
\begin{enumerate}
\item $fg\le \mathrm{id}_Y$ and $\mathrm{id}_X\le gf$,
\item left adjoints preserve all existing suprema and right adjoints preserve all existing infima,
\item
and  if $X,Y$ are complete lattices then each $f\colon  X \to Y$ preserving all suprema is a left adjoint, and each $g \colon  Y \to X$ preserving all infima is a right adjoint.
\end{enumerate}

Adjunctions determine a particularly close tie between two categories and ``occur almost
everywhere in many branches of
Mathematics'' (MacLane \cite[p. 103]{SM}). In particular, Galois adjunctions determine a certain connection between two posets and may describe many mathematical concepts. For instance, let
$f\colon X\to Y$ be a function between topological spaces and define the pair of assignments

\bigskip
\begin{center}
$\xymatrix@C=30pt@R=-7pt{& & & A \ar@/^.6pc/@{|..>}[r]_{}  &  \clr(f[A]) \\
\mbox{Closed}(X) \ar@/^1.5pc/[rr]^{f^{\to}}   & & \mbox{Closed}(Y) \ar@/^1.5pc/[ll]^{f^{\leftarrow}} & & \\
& & & \clr(f^{-1}[B]) &    B \ar@/^.6pc/@{|..>}[l]_{}\    }
$
\end{center}

\bigskip\nid
Then it is very easy to check that
\begin{quote}
{\em $f$ is continuous iff $(f^\to,f^\leftarrow)$ is an adjoint pair.}\quad (\cite{EL06,EL15})
\end{quote}

\medskip
Motivated by this basic example and by recent characterizations in the pointfree setting of {\em continuity} (that is, the property of a mapping to be localic; see \cite{EPP22}), we address in this article similar adjoint situations in pointfree topology (that is, in the category of locales and localic maps) that interchange images or preimages with closure or interior operators. We establish with them a number of characterizations for meet-preserving maps, localic maps, open maps (in a broad sense) and open localic maps between locales.

The main feature of the adjunctions we  are presenting here is that they are all
concerned with elementary ideas and basic concepts of localic topology: the use of the concrete language of sublocales and its technique simplifies the reasoning. There is, however, a delicate point: the complements
of closed or open sublocales have to be formed in the lattice of all sublocales
and not set-theoretically as in classical topology.

Our results may be extended to the setting in \cite{PPT,EPP22} (Heyting or implicative semilattices) but we do not pursue that direction in the present paper, we leave it for some future work. Our goal is to state these adjunctions just in the pointfree setting in order to better compare the pointfree notions of continuity and openness (as properties of meet-preserving maps) with their classical counterparts.

\medskip
Here is an overview of the paper. We begin with a brief account of the background involved here (Section 2). Our general reference for pointfree topology and lattice theory is \cite{PP12} (or the Appendix in \cite{PP21}) and the more recent \cite{PP22}. In Section 3, we describe four types of adjunctions with closure and interior operators that characterize respectively the preservation of meets, continuity, openness and
continuity+openness of a plain map between locales. In Section 4, we revisit the
open mapping theorem of Joyal-Tierney \cite{JT84}. In particular, we present a short proof for it that takes advantage of the concrete technique of sublocales.
In the last two sections (5 and 6), we deal with the commutative properties between preimages (and images) and closure or interior operators,
emphasizing the similarities and differences with the classical setting.

\section{Preliminaries} For the convenience of the reader, we recall here the notions and facts which will be of particular relevance in our
context.

\subsection{Frames.} A {\em frame}  is a complete lattice $L$ satisfying the distributivity law
\[(\tbigvee A\big)\wedge b
= \tbigvee\setof{a\wedge b}{a \in A}\]
for all $A\subseteq L$ and $b\in L$. A \emph{frame homomorphism} preserves all joins (including the bottom element 0 of the frame) and all finite meets (including the top element 1).

The distributivity law makes $(-)\wedge b$, for each $b$, a left adjoint; consequently a frame has a {\em Heyting structure}
with the Heyting operation $\to$ satisfying
\[
a\wedge b\leq c \qtq{iff} a\leq b\to c.
\]
A frame has {\em pseudocomplements} $a^* = a \to 0=\tbigvee\{b\in L\mid a\wedge b=0\}$. Recall the standard fact that
$a\leq a^{**}$.

 We will also use  simple Heyting  rules like (see e.g. \cite{PP12}):
\begin{enumerate}[(H1)]
\item $1\to a=a$, $a\to b=1$ iff $a\le b$.
\item $a\to\bim_{i\in I}b_i= \bim_{i\in I}(a\to b_i)$.
\item $a\to(b\to c)=(a\wedge b)\to c=b\to (a\to c)$.
\item $(\tbigvee_{i\in I}a_i)\to b= \bim_{i\in I}(a_i\to b)$.
\end{enumerate}

\subsection{Locales.}
A typical frame is the lattice $\Omega(X)$ of open sets of a topological space $X$, and for each continuous map $f \colon  X \to Y$ there is a  frame homomorphism $\Omega(f) = (U \mapsto f^{-1}[U]) \colon  \Omega(Y) \to \Omega(X)$. Thus we have a contravariant functor
\[
\Omega \colon  \TOP \to \frm,
\]
where $\TOP$ is  the category of topological spaces and continuous maps, and $\frm$ that of frames and frame homomorphisms. To make it covariant one considers the {\em category of locales} $\loc$, the dual category of $\frm$.

\subsection{Localic maps.}\label{locmaps}
It is of advantage to view  $\loc=\frm\op$ as a concrete category with the opposites of frame homomorphisms $h\colon M\to L$ represented by their right adjoints $f\colon  L \to M$, called {\em localic maps}. Being right adjoints, localic maps are meet-preserving maps. They are precisely the meet-preserving $f\colon L\to M$ that satisfy the following two conditions (see \cite{PP12} for more information):
\begin{enumerate}[(L1)]
\item
$f(a)=1\ \Rightarrow \ a=1$.
\item
$f(h(a)\to b)= a\to f(b)$ for all $a\in M$, $b\in L$.
\end{enumerate}

Indeed, (L1) is the condition on $f$ that corresponds to the fact that $h$ preserves the empty meet $1$, while (L2) is the condition on $f$ equivalent to the fact that
$h$ preserves binary meets.

It is usual to denote the left adjoint (when it exists) of a map $f$ by $f^*$. We shall do it from now on.

\subsection{Sublocales.} A {\em sublocale} of a locale $L$ is a subset $S\sue L$ that satisfies the following conditions  (\cite{PP12}):
\begin{enumerate}[(S1)]
\item
For every $M \subseteq S$ the meet $\tbigwedge M$ lies in $S$.
\item
For every $s\in S$ and every $x\in L$, $x \to s$ lies in $S$.
\end{enumerate}
A sublocale $S$ is itself a locale with meets computed as in $L$ (by (S1)) but with possibly different joins; by (S2), its
Heyting operation is the restriction of the $\to$ from $L$.

The system $\SL(L)$ of all sublocales of $L$ is a {\em coframe} (that is, the dual of frame) with a fairly transparent structure:
\begin{equation}\tag{2.4.1}\label{joinsub}
\tbigwedge S_i
= \tbigcap S_i
\qtq{and} \tbigvee S_i
= \setof{\tbigwedge M}{M\subseteq \tbigcup S_i}.
\end{equation}
The least sublocale $\tbigvee \emptyset  = \{1\}$ is designated by $\OS$ and referred to as the
{\em void sublocale}; the largest sublocale is, of course, $L$.

Being a coframe, $\SL(L)$ is naturally endowed with a co-Heyting operation, the {\em difference}
$R \smin S$ (\cite{FPP}) satisfying
\[
R\smin S\leq T \qtq{iff} R\leq S\vee T,
\]
and it has co-pseudocomplements (usually called {\em supplements}) $S^\#=L\smin S$ (the smallest $T$ such that $S\vee T=L$). In case $S$ is complemented, $S^\#$ is the complement of $S$ and we shall denote it by $S\co$.

\subsection{Open and closed sublocales.} With  $a\in L$ we associate the {\em open} and {\em closed}
sublocales
\[
\of(a)=\setof{x}{a\to x=x}=\setof {a\to x}{x\in L} \qtq{and} \cf(a)=\ur a=\{x\in L\mid x\ge a\}
\]
(for the equality in the first use (H3)).
For each $a\in L$,  $\of(a)$ and $\cf(a)$ are complements of each other in $\SL(L)$.
Moreover,
\[
\begin{aligned}
&\of(0)=\OS, \of(1)=L,\ \ \of(a)\cap\of(b)=\of(a\wedge b) \ \text{ and }\ \tbigvee\of(a_i)=\of(\tbigvee a_i),\\
&\cf(1)=\OS, \cf(0)=L,\ \ \cf(a)\vee\cf(b)=\cf(a\wedge b) \ \text{ and }\ \tbigcap\cf(a_i)=\cf(\tbigvee a_i).
\end{aligned}
\]

Similarly like in spaces we have the {\em closure}, the smallest closed sublocale containing $S$,
\begin{equation}
\clr S=\tbigcap\setof{\cf(a)}{S\sue\cf(a)}=\cf(\bim S). \tag{2.5.1}
\end{equation}

Using the fact that $S^{\#\#}\sue S$ we immediately see that $\of(a)\sue S$ iff $S^{\#}\sue\cf(a)$ and obtain the formula for the {\em interior} of $S$,
\begin{equation}
\intr S=\tbigvee\setof{\of(a)}{\of(a)\sue S}=\of(\tbigvee\{a\mid S^{\#}\sue\cf(a) \})=\of(\bim S^\#). \tag{2.5.2}
\end{equation}

We have to be careful here: since $\SL(L)$ is not a Boolean algebra in general, some of the usual formulas relating the interior and closure are no longer valid in locales. Indeed, one only has (see e.g. \cite{FPP})
\[\clr(L\smin A)=L\smin \intr A \quad \mbox{ and } \quad \intr(L\smin A)\supseteq L\smin\clr A.
\]

\subsection{Meet-preserving maps in locales}
 Let $f\colon L\to M$ be a (plain) map between locales such that
\begin{equation}\tag{2.6.1}\label{cond1}
\forall b\in M\ \exists a\in L\ \colon\ f^{-1}[\cf(b)]=\cf(a).
\end{equation}
Of course, the $a$ is necessarily unique and $f$ has a left adjoint $f^* \colon M\to L$ given by the assignment $b\mapsto a$:
\[y\le f(x)\Leftrightarrow f(x)\in\cf(y)\Leftrightarrow x\in f^{-1}[\cf(y)]=\cf(f^*(y))\Leftrightarrow f^*(y)\le x.\]
(In particular, $f$ and $f^*$ are order-preserving.)

Conversely, we also have:

\begin{propsub}
\label{galois}
Let $f\colon L\to M$ be a map between locales. If $f$ is a right adjoint, with left adjoint $f^*$, then $f^{-1}[\cf(b)]=\cf(f^*(b))$ for every  $b\in M$.
\end{propsub}

\begin{proof}
$x\in f^{-1}[\cf(b)]$ iff $ b\le f(x)$ iff $ f^*(b)\le x$ iff $ x\in \cf(f^*(b))$.
\end{proof}

Hence:

\begin{corosub}\label{corogalois}
For any locales  $L, M$,
a map $f\colon L\to M$ is a right adjoint  \textup(that is, preserves arbitrary meets\textup) if and only if condition \eqref{cond1} holds.
\end{corosub}

\subsection{Images and preimages.} \label{images}
Let $f\colon L\to M$ be a localic map and $S\sue L$ and $T\sue M$ sublocales.
Then the standard set theoretic image $f[S]$ is easily seen to be a sublocale of $M$. The standard $f^{-1}[T]$ is generally not a sublocale, but it is closed under meets and hence, by \eqref{joinsub}, there is the  largest sublocale contained in it, namely the join $\tbigvee\setof{S\in\SL(L)}{S\sue f^{-1}[T]}$. We denote it by \[f_{-1}[T]\] and call it the {\em localic preimage}
 of $T$ (as
opposed to the {\em set theoretic preimage} $f^{-1}[T]$). There is the obvious   adjunction
\[
f[S]\sue T \qtq{iff} S\sue f_{-1}[T].
\]
In particular, $f[-]$ preserves joins and $f_{-1}[-]$ preserves meets (=intersections) of sublocales. Furthermore, $f_{-1}[-]$ is a coframe homomorphism and hence $f[-]$ is a colocalic map (see \cite{PP12} for details). Moreover, $f_{-1}[\OS]=\OS$ (because already
$f^{-1}[\OS]=\OS$, by (L1)).

Localic preimages of closed (resp. open) sublocales are closed (resp. open), more precisely,
\begin{equation}\tag{2.7.1}
f_{-1}[\cf(b)]=\cf(f^*(b)) \quad\mbox{ and }\quad f_{-1}[\of(b)]=\of(f^*(b))
\end{equation}
where $f^*\colon M\to L$ stands for the left adjoint of $f$. The proof of the first identity is obvious; we include here a proof of the
second one, much shorter than the ones in the literature (e.g. \cite{PP08,PP12,PP21}):

By (L2), $\of(f^*(b))\subseteq f^{-1}[\of(b)]$ hence $\of(f^*(b))\subseteq f_{-1}[\of(b)]$.
On the other hand
\[f_{-1}[\of(b)]\cap f_{-1}[\cf(b)]=f_{-1}[\of(b)\cap \cf(b)]=f_{-1}[\OS]=\OS\]
hence $f_{-1}[\of(b)]\subseteq f_{-1}[\cf(b)]\co=\of(f^*(b))$.

\medskip
More generally, for any meet-preserving $f$ and any $T$ closed under meets, $f^{-1}[T]$ is closed under meets and hence
we have a largest sublocale $f_{-1}[T]$  contained in $f^{-1}[T]$. We will use this notation in this more general context and to avoid confusion
state that ``the localic preimage
$f_{-1}[T]$ makes sense''.

\section{Adjunctions via closure and interior}

\subsection{Type I: Characterizing the preservation of meets}

It will be of advantage to extend the closure operator on sublocales to general subsets $S$ of $L$.

Let $\cf L$ denote the set of all closed sublocales of $L$. This is a sub-coframe of the coframe $\mathsf{S}(L$).
For any {\bf subset} $S\subseteq L$, let $\clr S$ denote the closed sublocale $$\tbigcap\{\cf(a)\in\cf L\mid S\subseteq \cf(a)\}=\cf(\tbigvee\{a\in L\mid S\subseteq \cf(a)\})=\cf(\bim S).$$
This defines a map from $\mathcal P(L)$ to $\cf L$ with the properties of a closure operator:
\begin{itemize}
\item {\em inflationary}: $S\subseteq \clr S$.
\item {\em order-preserving}: $S\subseteq T\ \Rightarrow \ \clr S\subseteq \clr T$.
\item {\em idempotent}: $\clr(\clr S)=\clr S$.
\end{itemize}
Clearly, one has the equivalence
\[S \subseteq \clr T \Leftrightarrow \clr S \subseteq \clr T \mbox{ for every }T, S \subseteq L.\]

\medskip
For each plain map $f\colon L\to M$  between locales consider the following:

\bigskip
\begin{center}
$\xymatrix@C=45pt@R=-7pt{& & & \cf(a) \ar@/^.6pc/@{|..>}[r]_{}  &  \clr(f[\cf(a)]) \\
\cf L \ar@/^1.5pc/[rr]^{f_\cf^{\to}}   & \bot & \cf M \ar@/^1.5pc/[ll]^{f_\cf^{\leftarrow}} & & \\
& & & \clr(f^{-1}[\cf(b)]) &    \cf(b) \ar@/^.6pc/@{|..>}[l]_{}\    }
$
\\[4mm]
{\sf Adjoint pair I: Proposition \ref{p313}}
\end{center}

\bigskip
\begin{remasub} Note that $\cf(f(a))\subseteq \clr(f[\cf(a)])$ holds always. Moreover,
 $f[\cf(a)]\subseteq \cf(f(a))$ iff \[x\ge a\ \Rightarrow \ f(x)\ge f(a).\] Therefore $f[\cf(a)]\subseteq \cf(f(a))$ for every $a$ iff $f$ is order-preserving. Hence
  if $f$ is order-preserving then, in the first assignment, $\clr(f[\cf(a)])$ is always equal to $\cf(f(a))$. In this case, it then follows immediately that
  \begin{equation}
  \tag{3.1.1}
  \cf(f(a))\subseteq \cf(b)\mbox{ iff } f[\cf(a)]\subseteq \cf(b).
  \end{equation}

On the other hand,
$f[\cf(a)]\supseteq \cf(f(a))$ iff
\[y\ge f(a)\ \Rightarrow \ y=f(x) \mbox{ for some }x\ge a\] (that is, ontoness in $\cf(f(a))$ w.r.t. $\cf(a)$).
\end{remasub}

The proof of the following result is straightforward, we skip it.

\begin{lemmsub}\label{lemadj1}
 For a plain map $f\colon L\to M$, the following are equivalent:
\begin{enumerate}[\em(i)]
\item
The pair $(f_\cf^\to,f_\cf^{\leftarrow})$ is an adjoint pair.
\item $\clr(f[\cf(a)]) \subseteq \cf(b)$ iff $\cf(a)\subseteq \clr(f^{-1}[\cf(b)])$ for every $a\in L$ and $b\in M$.
\item $f[\cf(a)] \subseteq \cf(b)$ iff $\cf(a)\subseteq \clr(f^{-1}[\cf(b)])$ for every $a\in L$ and $b\in M$.
\item $\cf(a)\subseteq f^{-1}[\cf(b)]$ iff $\cf(a)\subseteq \clr(f^{-1}[\cf(b)])$ for every $a\in L$ and $b\in M$.
\end{enumerate}
Moreover, if $f$ is order-preserving then the above conditions are also equivalent to each one of the following conditions:
\medskip
\begin{enumerate}[\em(i)]
\item[\em(v)] $f(a)\ge b$ iff $\cf(a)\subseteq \clr(f^{-1}[\cf(b)])$ for every $a\in L$ and $b\in M$.
\item[\em(vi)] $f(a)\ge b$ iff $a\in \clr(f^{-1}[\cf(b)])$ for every $a\in L$ and $b\in M$.
\item[\em(vii)] $a\in f^{-1}[\cf(b)]$ iff $a\in \clr(f^{-1}[\cf(b)])$ for every $a\in L$ and $b\in M$.
\item[\em(viii)] $a\in \clr(f^{-1}[\cf(b)])\ \Rightarrow \ a\in f^{-1}[\cf(b)]$ for every $a\in L$ and $b\in M$.
\end{enumerate}
\end{lemmsub}

\begin{propsub}\label{p313}
Let $f\colon L\to M$ be a plain map between locales.
The pair $(f_\cf^\to,f_\cf^{\leftarrow})$ is an adjoint pair if and only if $f$ preserves arbitrary meets.
\end{propsub}

\begin{proof}
$\Leftarrow$: By Cor. \ref{corogalois}, condition (iv) of the lemma is trivially satisfied.

\medskip
\noindent
$\Rightarrow$: Let $\cf(a)=\clr(f^{-1}[\cf(b)])$. It follows from Lemma \ref{lemadj1}(iv) that $\clr(f^{-1}[\cf(b)])\subseteq f^{-1}[\cf(b)]$ and thus, by Cor. \ref{corogalois}, $f$ preserves arbitrary meets.
\end{proof}

\begin{corosub}\label{coradj1}
For a plain map $f\colon L\to M$, the following are equivalent:
\begin{enumerate}[\em(i)]
\item $(f_\cf^\to,f_\cf^{\leftarrow})$ is an adjoint pair.
\item $f$ preserves arbitrary meets.
\item $\forall b\in M\ \exists a\in L \colon\ f^{-1}[\cf(b)]=\cf(a)$.
\end{enumerate}
And when these conditions hold then $f$ has a left adjoint $f^*$ satisfying
\begin{enumerate}[\em(i)]
\item[\em(iv)] $\forall b\in M, \ f^{-1}[\cf(b)]=\cf(f^*(b))$.
\end{enumerate}
\end{corosub}

Now, we can rephrase the characterization of localic maps from \cite[Cor. 4.8]{EPP22} (cf. \cite[Thm. 6.5.2]{PP22}) in the following way:

\begin{theosub}
A plain map $f\colon L\to M$ between locales is a localic map if and only if
\begin{enumerate}[\em(1)]
\item
$(f_\cf^\to,f_\cf^{\leftarrow})$ is an adjoint pair,
\item $f_\cf^\leftarrow[\OS]=\OS$, and
\item $f_\cf^\leftarrow[\cf(b)]\co\subseteq f^{-1}[\of(b)]$ for every $b\in M$.
\end{enumerate}
\end{theosub}

\begin{proof} The result in \cite{EPP22} asserts that $f$ is a localic map iff
\begin{enumerate}[(a)]
\item $f^{-1}[\OS]=\OS$,
\item
$f^{-1}[A]$ is closed for every closed sublocale $A$ of $M$, and
\item $f^{-1}[U]\supseteq f^{-1}[U\co]\co$ for every open sublocale $U$ of $M$
\end{enumerate}
so it suffices to check that (1)\&(2)\&(3) $\Leftrightarrow$ (a)\&(b)\&(c).

Condition (a) is clearly equivalent to (2) and, by  Cor. \ref{coradj1}, (b) is equivalent to (1). Then, under (b) or (1), we have always $f_\cf^\leftarrow[\cf(b)]=\clr(f^{-1}[\cf(b)])=f^{-1}[\cf(b)]$ and the equivalence (c)$\Leftrightarrow$(3) follows.
\end{proof}

\subsection{Type II:  characterizing continuity }

Let $L$ be a complete lattice. From now on we shall refer to subsets of $L$ that are closed under meets as {\em meet-subsets} of $L$.
The system of all meet-subsets in $L$
will be denoted by
$$
\MC(L).
$$
Intersections of meet-subsets are meet-subsets. Consequently,
$\MC(L)$
is a  complete lattice. The joins in $\MC(L)$ are also given by the formula \eqref{joinsub} and
its bottom  is again
$
\OS=\set{1}.
$

For each frame $L$, let $\of L$ denote the set of all open sublocales of $L$ and,
for any {\bf subset} $S\subseteq L$, let $\intr S$ denote the open sublocale $$\tbigvee\{\of(a)\in\of L\mid \of(a)\subseteq S\}=\of(\tbigvee\{a\in L\mid \of(a)\subseteq S\}).$$
This defines a mapping from $\mathcal P(L)$ to $\of L$ with two of the properties  of interior operators:
\begin{itemize}
\item {\em order-preserving}: $S\subseteq T\ \Rightarrow \ \intr S\subseteq \intr T$.
\item {\em idempotent}: $\intr(\intr S)=\intr S$.
\end{itemize}
However,
$\intr S$ is not necessarily contained in $S$ because the join above is taken in $\SL(L)$ (e.g. take any $S$ that does not contain the top element 1). Hence, we may not have the implication $\of(a)\subseteq \intr S\ \Rightarrow \ \of(a)\subseteq S$.
Nevertheless, for $$S^\wedge=\tbigcap\{T\in \MC(L)\mid T\supseteq S\}\in\MC(L)$$ we have the following:

\begin{propsub} For every $S\subseteq L$,
$\intr S\subseteq S^\wedge$.
\end{propsub}

\begin{proof}
Let $T\supseteq S$, $T\in\MC(L)$. We need to check that $\intr S\subseteq T$. But $$\intr S =\{\tbigwedge A\mid A\subseteq \tbigcup\{\of(a)\mid \of(a)\subseteq S\}\}$$ and each of those $A$ is contained in $S\subseteq T$ hence $\tbigwedge A\in T$.
\end{proof}

This proof also shows that
$$
\mbox{\em $\intr S\subseteq S$ whenever $S\in\MC(L)$}
$$
and thus $$\of(a)\subseteq S \ \mbox{ iff } \ \of(a)\subseteq \intr S\mbox{ \em for any }S\in\MC(L).$$

\bigskip
The following result is obvious.

\begin{lemmsub}\label{meetsets}
Let $f\colon L\to M$ be a meet-preserving map between locales and let $S\in \MC(L)$ and $T\in\MC(M)$. Then  $f[S]\in\MC(M)$ and $f^{-1}[T]\in\MC(L)$ and we have again an adjunction
\begin{center}
$\xymatrix@C=35pt{
\MC(L) \ar@/^1.2pc/[rr]^{f[-]}   & \bot &   \MC(M). \ar@/^1.2pc/[ll]^{f^{-1}[-]} }$
\end{center}
\end{lemmsub}

\begin{lemmsub}\label{lemmf0}
Let $f\colon L\to M$ be a meet-preserving map between locales. Then
$f^{-1}[\OS]=\OS$ if and only if  $f^{-1}[\cf (b)]\subseteq(\intr (f^{-1}[\of(b)]))\co$ for every $b\in M$.
\end{lemmsub}

\begin{proof}
``$\Rightarrow$'': Since $f^{-1}[\of(b)]$ is a meet-subset we have
\begin{align*}
    f^{-1}[\cf (b)]\cap \intr (f^{-1}[\of(b)]) & \subseteq f^{-1}[\cf (b)]\cap f^{-1}[\of(b)] \\
    &=f^{-1}[\cf (b) \cap \of(b)]=\OS.
\end{align*}
We may then conclude that $f^{-1}[\cf (b)]\subseteq(\intr (f^{-1}[\of(b)]))\co$.

\smallskip
\nid
``$\Leftarrow$'': Take just the case $b=1$:
$$f^{-1}[\OS]=f^{-1}[\cf(1)]\subseteq (\intr(f^{-1}[\of(1)]))\co=\OS.\qedhere$$
\end{proof}

Now, we can rephrase again the characterization of localic maps from \cite[Cor. 4.8]{EPP22} (cf. \cite[Thm. 6.5.2]{PP22}) that asserts that a map $f\colon L\to M$ is a localic map if and only if the following conditions hold:
\begin{enumerate}[(a)]
\item $f^{-1}[\OS]=\OS$.
\item $f^{-1}[\cf (b)]$ is closed for every $b\in M$.
\item $f^{-1}[\of(b)]\supseteq f^{-1}[\cf(b)]\co$ for every $b\in M$.
\end{enumerate}

Indeed, it follows from (c) (and (b)) that $\intr (f^{-1}[\of(b)])\supseteq f^{-1}[\cf(b)]\co$. Moreover, since
(b) implies that $f$ is a meet-preserving map, Lemma \ref{lemmf0} guarantees the other inclusion and we may conclude that $f$ is a localic map if and only if
\begin{enumerate}[(1)]
\item
$f^{-1}[\cf (b)]$ is closed for every $b\in M$, and
\item $f^{-1}[\cf (b)]=(\intr (f^{-1}[\of(b)]))\co$ for every $b\in M$.
\end{enumerate}
Finally, since $(\intr (f^{-1}[\of(b)]))\co$ is a closed sublocale, we get immediately the following characterization of localic maps:

\begin{propsub}\label{theofloccaract}
    A plain map $f\colon L\to M$ between locales  is a localic map if and only if
    $$(\intr(f^{-1}[\of(b)]))\co = f^{-1}[\cf(b)]\quad\mbox{for every }b\in M.$$
\end{propsub}

Moreover, we have:
\begin{propsub}
Let $f\colon L\to M$ be a meet-preserving map between locales. Then $\intr(f^{-1}[\OS])=\OS$ if and only if
\[\intr(f^{-1}[\cf(b)]) \subseteq (\intr(f^{-1}[\of(b)]))\co\quad\mbox{for every }b\in M.\]
\end{propsub}

\begin{proof}
``$\Rightarrow$'': Since $f$ is a meet-preserving map, $f^{-1}[\of(b)]$ and $f^{-1}[\cf(b)]$ are meet-subsets, and therefore
   $$\intr(f^{-1}[\of(b)])\cap \intr (f^{-1}[\cf(b)]) \subseteq f^{-1}[\of(b)] \cap f^{-1}[\cf(b)].$$
   Then
   \begin{align*}
       \intr(f^{-1}[\of(b)])\cap \intr (f^{-1}[\cf(b)]) &\subseteq \intr (f^{-1}[\of(b)] \cap f^{-1}[\cf(b)]) \\
       & =\intr f^{-1}[\OS]=\OS
   \end{align*}
   hence $\intr(f^{-1}[\cf(b)]) \subseteq (\intr(f^{-1}[\of(b)]))\co$.

   \smallskip
   \nid
   $\Leftarrow$: It suffices to consider the case $b=1$:
   $$\intr(f^{-1}[\OS])=\intr(f^{-1}[\cf(1)]) \subseteq (\intr(f^{-1}[\of(1)]))\co=\OS.\qedhere$$
\end{proof}

\bigskip
We now consider, for each (plain) map $f\colon L\to M$  between locales, the mappings $f_\of^{\to}$ and $f_\of^{\leftarrow}$ given by

\bigskip
\begin{center}
$\xymatrix@C=45pt@R=-7pt{& & & \of(a) \ar@/^.6pc/@{|..>}[r]_{}  &  (\clr(f[\cf(a)]))\co  \\
\of L \ar@/^1.5pc/[rr]^{f_\of^{\to}}   & \top & \of M \ar@/^1.5pc/[ll]^{f_\of^{\leftarrow}} & & \\
& & & \intr(f^{-1}[\of(b)])  &    \of(b) \ar@/^.6pc/@{|..>}[l]_{}\    }
$
\\[4mm]
{\sf Adjoint pair II: Theorem \ref{t327}}
\end{center}

\bigskip
\begin{remasub}
 If $f$ is localic then the largest sublocale contained in $f^{-1}[\of(b)]$ is $f_{-1}[\of(b)]=\of(f^*(b))$ hence $f_\of^{\leftarrow}(\of(b))=\of(f^*(b))$.
\end{remasub}

\begin{theosub}\label{t327}
Let $f\colon L\to M$ be an order-preserving map between locales.
The pair $(f_\of^\leftarrow,f_\of^{\to})$ is an adjoint pair if and only if $f$ is a localic map.
\end{theosub}

\begin{proof}
For each $a\in L$ and $b\in M$,
\begin{align*} \Bigl(\intr(f^{-1}[\of(b)]) \subseteq \of(a) &\ \Leftrightarrow \ \of(b)\subseteq (\clr(f[\cf(a)]))\co\Bigr)\\
\mbox{iff } \ \    \Bigl(\cf(a)\subseteq (\intr(f^{-1}[\of(b)]))\co &\ \Leftrightarrow \ \clr(f[\cf(a)]) \subseteq \cf(b)\Bigr)\\
\mbox{iff } \ \    \Bigl(\cf(a)\subseteq (\intr(f^{-1}[\of(b)]))\co &\ \Leftrightarrow \ f[\cf(a)] \subseteq \cf(b)\Bigr).
\end{align*}
Since $f$ is order-preserving this is further equivalent to
\begin{align*} \Bigl(\cf(a)\subseteq (\intr(f^{-1}[\of(b)]))\co &\ \Leftrightarrow \ f(a) \in \cf(b)\Bigr), \mbox{ that is},\\
\mbox{iff } \ \    \Bigl(a\in (\intr(f^{-1}[\of(b)]))\co &\ \Leftrightarrow \ a\in f^{-1}[\cf(b)]\Bigr).
\end{align*}
The conclusion follows then immediately from Prop. \ref{theofloccaract}.
\end{proof}

Compare Prop.~\ref{theofloccaract} with this theorem: in the former we have the equality $(\intr(f^{-1}[\of(b)]))\co= f^{-1}[\cf(b)]$
while in the latter $(\intr(f^{-1}[\of(b)]))\co$ is the greatest closed sublocale contained in $f^{-1}[\cf(b)]$.
Next result is an immediate consequence of the theorem.

\begin{corosub}
A plain map $f\colon L\to M$ between locales  is a localic map if and only if $f$ is order-preserving and $(f_\of^\leftarrow,f_\of^{\to})$ is an adjoint pair.
\end{corosub}

\subsection{Type III: characterizing openness}\label{subopen}

As another variant, replace $f_\of^{\rightarrow}$ by the following $f_\of^{\Rightarrow}$:

\bigskip
\begin{center}
$\xymatrix@C=45pt@R=-7pt{& & & \of(a) \ar@/^.6pc/@{|..>}[r]_{}  &  \intr(f[\of(a)]) \\
\of L \ar@/^1.5pc/[rr]^{f_\of^{\Rightarrow}}   & \bot & \of M \ar@/^1.5pc/[ll]^{f_\of^{\leftarrow}} & & \\
& & & \intr(f^{-1}[\of(b)]) &    \of(b) \ar@/^.6pc/@{|..>}[l]_{}\    }
$
\\[4mm]
{\sf Adjoint pair III: Theorem \ref{theoadj3}}
\end{center}

\bigskip
Open continuous maps are naturally modelled in pointfree topology as {\em open localic maps}, that is,
localic maps $f \colon L \to M$ such that the image $f[\of(a)]$ of every open sublocale is open.
They are characterized by the celebrated Joyal-Tierney Theorem (\cite{JT84}):

\begin{quote}
{\em
A localic map $f \colon L \to M$ is open iff the adjoint frame homomorphism
$f^*\colon M\to L$ is a  complete Heyting homomorphism \textup(i.e. if it
preserves also arbitrary meets and the Heyting operation\textup).}
\end{quote}

Here we intend to study the openness property for general maps, not necessarily localic, so we will speak about {\em open maps} in a broad sense as (plain) maps $f\colon L\to M$ between locales
such that the image $f[\of(a)]$ of every open sublocale is still open.

\begin{lemmsub}\label{lemadj3}
 The following are equivalent for a  meet-preserving map $f\colon L\to M$ between locales:
\begin{enumerate}[\em(i)]
\item
The pair $(f_\of^\Rightarrow,f_\of^{\leftarrow})$ is an adjoint pair.
\item $\intr(f[\of(a)]) \subseteq \of(b)$ iff $\of(a)\subseteq \intr(f^{-1}[\of(b)])$, for every $a\in L$ and $b\in M$.
\item $\intr(f[\of(a)]) \subseteq \of(b)$ iff $\of(a)\subseteq f^{-1}[\of(b)]$, for every $a\in L$ and $b\in M$.
\item $\intr(f[\of(a)]) \subseteq \of(b)$ iff $f[\of(a)]\subseteq \of(b)$, for every $a\in L$ and $b\in M$.
\end{enumerate}
\end{lemmsub}
\begin{proof}
    The equivalences (i) $\Leftrightarrow$ (ii) and (iii) $\Leftrightarrow$ (iv) are trivial.

\nid
    (ii) $\Leftrightarrow$ (iii): Since $f$ is a meet-preserving map, $f^{-1}[\of(b)]$ is a meet-subset (recall Lemma \ref{meetsets}). Therefore $\of(a)\subseteq \intr(f^{-1}[\of(b)])$ iff $\of(a)\subseteq f^{-1}[\of(b)]$.
\end{proof}

\begin{theosub}\label{theoadj3}
Let $f\colon L\to M$ be a meet-preserving map. The pair $(f_\of^\Rightarrow,f_\of^{\leftarrow})$ is an adjunction if and only if $f$ is open.
\end{theosub}

\begin{proof}
``$\Rightarrow$'': Consider an $a\in L$ and let $\of(b)=\intr(f[\of(a)])$. Since $f$ is meet-preserving, $f[\of(a)]\in\MC(M)$ (by Lemma \ref{meetsets}) and therefore $\intr(f[\of(a)])\subseteq f[\of(a)]$. On the other hand, it follows from condition (iv) in the Lemma that $f[\of(a)]\subseteq\intr(f[\of(a)])$. Hence $f$ is open.

\smallskip
\nid
``$\Leftarrow$'': If $f[\of(a)]=\intr(f[\of(a)])$, then condition (iv) of the Lemma is trivially satisfied.
\end{proof}

Combining adjunctions II and III (in Theorems \ref{t327} and \ref{theoadj3}) we get:

\begin{corosub}
An order-preserving map $f\colon L\to M$ is an open localic map if and only if $$f_\of^\Rightarrow\dashv f_\of^\leftarrow \dashv f_\of^\to.$$
\end{corosub}

\medskip
\subsection{Type IV: characterizing openness combined with continuity}

Finally, consider the mappings

\bigskip
\begin{center}
$\xymatrix@C=45pt@R=-7pt{& & & \cf(a) \ar@/^.6pc/@{|..>}[r]_{}  &  (\intr(f[\of(a)]))\co \\
\cf L \ar@/^1.5pc/[rr]^{f_\cf^{\Rightarrow}}   & \top & \cf M \ar@/^1.5pc/[ll]^{f_\cf^{\leftarrow}} & & \\
& & & \clr(f^{-1}[\cf(b)]) &    \cf(b) \ar@/^.6pc/@{|..>}[l]_{}\    }
$
\\[4mm]
{\sf Adjoint pair IV: Theorem \ref{theoadj4}}
\end{center}

\bigskip
\begin{remasub}
If $f$ is meet-preserving then  $f_\cf^{\leftarrow}[\cf(b)]=f^{-1}[\cf(b)]=\cf(f^*(b))$, where $f^*$ stands for the left adjoint of $f$.
\end{remasub}

Let $f\colon L\to M$ be a meet-preserving map between locales. For each $a\in L$, there is a $b\in M$ such that $\intr(f[\of(a)])=\of(b)$, and the $b$ is obviously unique. This assignment defines a map
$\varphi\colon L\to M$ such that $\intr(f[\of(a)])\co=\cf(\varphi(a))$. The proof of the following result is straightforward.

\begin{lemmsub}\label{lemadj4}
 The following are equivalent for a meet-preserving $f\colon L\to M$:
\begin{enumerate}[\em (i)]
\item $(f_\cf^{\leftarrow},f_\cf^\Rightarrow)$ is a an adjunction.
\item $\clr(f^{-1}[\cf(b)])\subseteq \cf(a)$ iff $\cf(b)\subseteq  (\intr(f[\of(a)]))\co$, for all $a\in L$ and $b\in M$.
\item $\cf(f^*(b))\subseteq \cf(a)$ iff $\cf(b)\subseteq \cf(\varphi(a))$, for all $a\in L$ and $b\in M$.
\item $a\le f^*(b)$ iff $\varphi(a)\le b$, for all $a\in L$ and $b\in M$.
\end{enumerate}
\end{lemmsub}

\begin{theosub}\label{theoadj4}
Let $f\colon L\to M$ be a meet-preserving map between locales. The pair $(f_\cf^{\leftarrow},f_\cf^\Rightarrow)$ is an adjunction if and only if $f$ is an open localic map.
\end{theosub}

\begin{proof} Suppose that $(f_\cf^{\leftarrow},f_\cf^\Rightarrow)$ is an adjunction. By condition (ii) of the lemma, for $a=1$ and $b=1$,
$$f^{-1}[\OS]\subseteq \cf(1) \  \mbox{ iff } \ \cf(1) \subseteq (\intr f[\of(1)])\co.$$ The right-hand side is obviously true so $f^{-1}[\OS]\subseteq \cf(1)=\OS$, that is, $$f(a)=1\ \Rightarrow \ a=1$$ which is precisely
condition (L1) in the definition of a localic map (\ref{locmaps}).
Moreover, from condition (iv) of the lemma we know that $\varphi$ is left adjoint to $f^*$. In particular, $f^*$ preserves binary meets, that is, $f$ satisfies condition (L2) of \ref{locmaps}.
In conclusion, $f$ is a localic map.

Then $\clr(f^{-1}[\cf(b)])=f_{-1}[\cf(b)]=f_{-1}[\of(b)]\co$ and
\[
\clr(f^{-1}[\cf(b)])\subseteq \cf(a)  \Leftrightarrow  f_{-1}[\of(b)]\co\subseteq \cf(a)  \Leftrightarrow \of(a) \subseteq f_{-1}[\of(b)] \Leftrightarrow f[\of(a)] \subseteq \of(b).
\]
Hence
condition (ii) of Lemma \ref{lemadj4} can be rephrased as
\[f[\of(a)] \subseteq \of(b) \ \mbox{ iff } \ \intr(f[\of(a)]) \subseteq \of(b)\]
which is precisely the openness condition on $f$.

Conversely, let $f$ be an open localic map. By Thm. \ref{theoadj3} and Lemma \ref{lemadj3}(ii), we have, for each $a\in L$ and $b\in M$,
\[\intr(f[\of(a)]\subseteq \of(b)\mbox{ iff }\of(a)\subseteq \intr(f^{-1}[\of(b)]),\]
that is,
\begin{equation}\tag{$*$}
\cf(b)\subseteq (\intr(f[\of(a)])\co\mbox{ iff }\intr(f^{-1}[\of(b)])\co\subseteq \cf(a).
\end{equation}
But $f$ is localic hence $\intr(f^{-1}[\of(b)])\co=f^{-1}[\cf(b)]$ (recall Prop. \ref{theofloccaract}) and ($*$) transforms to
\[\cf(b)\subseteq (\intr(f[\of(a)])\co\mbox{ iff }f^{-1}[\cf(b)]\subseteq \cf(a),\]
which asserts that $(f_\cf^{\leftarrow},f_\cf^\Rightarrow)$ is an adjunction by Lemma \ref{lemadj4}.
\end{proof}

Finally, combining adjunctions I and IV (Proposition \ref{p313} and Thm. \ref{theoadj4}) we obtain:

\begin{corosub}
A plain map $f\colon L\to M$ is an open localic map if and only if $$f_\cf^\to \dashv f_\cf^\leftarrow \dashv f_\cf^\Rightarrow.$$
\end{corosub}

\medskip
\section{Open localic maps revisited}

Recall the Joyal-Tierney open mapping theorem from \ref{subopen}. In this section, we will present a most direct proof of it, using only basics of the language of sublocales.
Besides the basic properties about images and preimages of localic maps, we only need the following property of localic maps:

\begin{lemm}\label{lem1}
Let $f\colon L\to M$ be a localic map with left adjoint $f^*$. For any $a\in L$ and $b\in M$, $$f[\of(a)\cap\of(f^*(b))]=f[\of(a)]\cap \of(b).$$
\end{lemm}

\begin{proof}
The inclusion `$\subseteq$' is clear since $f[\of(f^*(b))]=ff_{-1}[\of(b)]\subseteq \of(b)$.

Conversely, let $y\in f[\of(a)]\cap \of(b)$. Then $y=b\to y$ and $y=f(a\to x)$ for some $x\in L$. Using (H3), we get
$$y=b\to y=b\to f(a\to x)=f(f^*(b)\to (a\to x))=f((a\wedge f^*(b))\to x)$$
where $(a\wedge f^*(b))\to x\in \of(a\wedge f^*(b))=\of(a)\cap \of(f^*(b))$.
\end{proof}

\begin{theo}\label{JTtheo}
The following are equivalent for a localic map $f\colon L\to M$:
\begin{enumerate}[\em(i)]
\item $f$ is open.
\item $f^*\colon M\to L$ is a complete Heyting homomorphism.
\item $f^*$ admits a left adjoint
$f_{!}$ that satisfies the identity
$$
f_!(a\wedge f^*(b))=f_!(a)\wedge b\mbox{ for all }a\in L\mbox{ and }b\in M.
$$
\item $f^*$ admits a left adjoint $f_{!}$  that satisfies the identity
$$
f(a\rightarrow f^*(b))=f_!(a)\rightarrow b\mbox{ for all }a\in L\mbox{ and }b\in M.
$$
\end{enumerate}
\end{theo}

\begin{proof}
(i)$\Rightarrow$(iii): Suppose that for any $a\in L$ there is some $b\in M$ such that $f[\of(a)]=\of(b)$. Of course, $b$ is unique and we have a map $f_!\colon L\to M$.
Since $f[\of(a)]\subseteq \of(c)$ iff $\of(a)\subseteq f_{-1}[\of(c)]$, and $f_{-1}[\of(c)]=\of(f^*(c))$, we conclude that $f_!(a)\le c$ iff $a\le f^*(c)$ and thus that $(f_!,f^*)$ is an adjoint pair.

Finally, the Frobenius identity $f_!(a\wedge f^*(b))=f_!(a)\wedge b$ follows from Lemma \ref{lem1}.
 Indeed,
\begin{align*}
\of(f_!(a\wedge f^*(b)))&=f[\of(a\wedge f^*(b))]=f[\of(a)\cap \of(f^*(b))]=\\&=f[\of(a)]\cap \of(b)=\of(f_!(a))\cap \of(b)=\of(f_!(a)\wedge b).
\end{align*}

\smallskip
\noindent
(iv)$\Rightarrow$(i): It suffices to check that $f[\of(a)]=\of(f_!(a))$ for every $a\in L$. The inclusion $\of(f_!(a))\subseteq f[\of(a)]$ holds by the hypothesis. On the other hand, from $f^*f_!(a)\ge a$ it follows that $\of(a)\subseteq \of(f^*(f_!(a)))=f_{-1}[\of(f_!(a))]$, that is, $f[\of(a)]\subseteq \of(f_!(a))$.

The equivalences (ii)$\Leftrightarrow$(iii)$\Leftrightarrow$(iv) are proved in \cite[Prop. 7.3]{PP08}, using a basic property on adjunctions. For the sake of completeness, we include the proof here:

\smallskip
\noindent
(ii)$\Leftrightarrow$(iii):
The condition in (iii) means that, for every $b\in M$, the diagram
\begin{equation*}
\xymatrix@+1pc{L\ar[r]^{f_{!}}\ar[d]_{f^*(b)\wedge(-)}&M\ar[d]^{b\wedge(-)}\\
L\ar[r]_{f_{!}}&M}
\end{equation*}
commutes. This is equivalent to saying that the corresponding square of right adjoints
\begin{equation*}
\xymatrix@+1pc{L&M\ar[l]_{f^*}\\
L\ar[u]^{f^*(b)\rightarrow(-)}&M\ar[l]^{f^*}\ar[u]_{b\rightarrow(-)}}
\end{equation*}
commutes, which is precisely the Heyting property for $f^*$.

\smallskip
\noindent
(iii)$\Leftrightarrow$(iv):
The condition in (iii) means also that for every $a\in L$ the square
\begin{equation*}
\xymatrix@+1pc{L\ar[d]_{a\wedge(-)}&M\ar[l]_{f^*}\ar[d]^{f_{!}(a)\wedge(-)}\\
L\ar[r]_{f_{!}}&M}
\end{equation*}
commutes. Again this  is so iff the corresponding square of right adjoints
\begin{equation*}
\xymatrix@+1pc{L\ar[r]^{f}&M\\
L\ar[u]^{a\rightarrow(-)}&M\ar[u]_{f_{!}(a)\rightarrow(-)}\ar[l]^{f^*}}
\end{equation*}
commutes, which is precisely the condition in (iv).
\end{proof}

\begin{remas}\label{remasopen}
(1)
Let $f\colon L\to M$ be a localic map. Then
\begin{align*}
\varphi \dashv f^* \ & \Leftrightarrow \
\bigl(\of(\varphi(a))\subseteq \of(b)\mbox{ iff }\of(a)\subseteq \of( f^*(b))\bigr) \\
& \Leftrightarrow \ \bigl(\of(\varphi(a))\subseteq \of(b)\mbox{ iff }\of(a)\subseteq f_{-1}[\of(b)]\bigr)\\
& \Leftrightarrow \ \bigl(\of(\varphi(a))\subseteq \of(b)\mbox{ iff }f[\of(a)]\subseteq \of(b)\bigr)
\end{align*}
and we conclude that
the existence of a left adjoint $\varphi$ of $f^*$ is equivalent to the condition
\begin{quote}
{\em For each $a\in L$ there is a unique smallest $b\in M$ such that $f[\of(a)]\subseteq \of(b)$,
}
\end{quote}
that is,
\begin{quote}
{\em For each $a\in L$ there is a unique smallest $b\in M$ such that $\of(a)\hookrightarrow L\stackrel{f}{\to} M$ factorizes through $\of(b)\hookrightarrow M$:

\begin{center}
$\xymatrix@C=30pt@R=25pt{L \ar[r]^{f} & M  \\
\of(a)\ar@{^{(}->}[u] \ar[r]^{f}   & \of(b)\ar@{^{(}->}[u]}
$
\end{center}
}
\end{quote}

\smallskip
\nid (2) It might be worth pointing out an important result from \cite{PP21} that shows that for a subfit locale $M$ (a very mild separation property on locales, weaker than $T_1$ in the setting of spaces), any complete frame
homomorphism $h=f^*\colon M\to L$ preserves automatically the Heyting operation (see \cite[II.3.5.1]{PP21} for a proof) and hence one has the following:

\medskip
\begin{quote}
{\bf Proposition.} {\em Let $M$ be subfit. Then a localic map $f \colon L\to M$ is open if and only if its adjoint
frame homomorphism $f^*\colon M\to L$ is a complete lattice homomorphism.
}
\end{quote}
\end{remas}

\bigskip
We conclude this section with a useful property of open localic maps.

\begin{prop}\label{hskeletal2}
Let $f\colon L\to M$ be an open localic map and $T\in\SL(M)$. The least sublocale of $L$ that contains $f^*[T]$ is contained in $f^{-1}[T]$ hence in $f_{-1}[T]$. In particular, $f^*[T]\subseteq
f_{-1}[T]$.
\end{prop}

\begin{proof} Let $S$ be the meet-closure of the set
$$\{a\to f^*(t))\mid a\in L, t\in T\}.$$
$S$ is a sublocale of $L$. Indeed, for each $b\in L$ and $\bim_{i\in I}(a_i\to f^*(t_i))\in S$, $$b\to (\bim_{i\in I}(a_i\to f^*(t_i)))=\bim_{i\in I}(b\to (a_i\to f^*(t_i)))=
\bim_{i\in I}((b\wedge a_i) \to f^*(t_i))\in S.$$
Hence it is clearly the least sublocale of $L$ that contains $f^*[T]$. Moreover, it is contained in $f^{-1}[T]$. In fact, by Thm. \ref{JTtheo}(iv),
$$f(\bim_{i\in I}(a_i\to f^*(t_i)))=\bim_{i\in I}f(a_i\to f^*(t_i))=\bim_{i\in I}(f_{!}(a_i)\to t_i)\in T.\qedhere$$
\end{proof}

\medskip
\section{Commutativity of preimages with closure and interior}

Trivially, localic maps  $f\colon L\to M$ satisfy the following containments for any sublocale (indeed, meet-subset) $T\subseteq M$:
\begin{enumerate}[(C1)]
\item $f_{-1}[\intr T]\subseteq \intr{f_{-1}[T]}$.
\item  $\clr(f_{-1}[T])\subseteq f_{-1}[\clr T]$.
\end{enumerate}
In this section, we investigate when the converse containments hold.

The converse to (C1) is precisely openness:

\begin{prop}
A meet-preserving $f\colon L\to M$ is open if and only if $$\intr{f_{-1}[T]}\subseteq f_{-1}[\intr T]\quad\mbox{ for every }T\in\MC(M).$$
\end{prop}

\begin{proof}
\smallskip
``$\Rightarrow$'': Let $f$ be open. It suffices to show that
$$\of(a)\subseteq f_{-1}[T] \ \Rightarrow \ \of(a)\subseteq f_{-1}[\intr T].$$
So let $\of(a)\subseteq f_{-1}[T] $. Then $f[\of(a)]\subseteq T$ hence $f[\of(a)]\subseteq \intr T$ (since $f[\of(a)]$ is open). Finally,
$\of(a)\subseteq f_{-1}f[\of(a)]\subseteq f_{-1}[\intr T]$.

%

\smallskip
\nid
``$\Leftarrow$'': Let $T=f[\of(a)]$. By Lemma \ref{meetsets}, $T$ is a meet-subset hence we have always $\intr T\subseteq T$ and, moreover, $$f_{-1}[\intr T]\supseteq\intr(f_{-1}[T])=\intr(f_{-1}f[\of(a)])\supseteq \of(a)$$ and thus
$T=f[\of(a)]\subseteq \intr T$.
\end{proof}

For localic maps we have:

\begin{coro}\label{propfopen}
Consider the following conditions about a localic map  $f\colon L\to M$:
\begin{enumerate}[\em (a)]
\item $f$ is open.
\item $\intr(f_{-1}[T])=f_{-1}[\intr T]$ for every sublocale $T$ of $M$.
\item $\clr(f_{-1}[T])=f_{-1}[\clr T]$ for every sublocale $T$ of $M$.
\end{enumerate}
Then $(a)\ \Leftrightarrow \ (b) \Rightarrow \ (c).$
\end{coro}

\begin{proof}
(a)$\Rightarrow$(b):  By the Proposition and (C1).

\smallskip
\nid
(b)$\Rightarrow$(a): $T=f[\of(a)]$ is now a sublocale and we have $$f_{-1}[\intr T]=\intr(f_{-1}[T])=\intr(f_{-1}f[\of(a)])\supseteq \of(a)$$ and thus
$T=f[\of(a)]\subseteq  \intr T$.

\smallskip
\nid
(a)$\Rightarrow$(c): By (C2) it suffices to check the inclusion ``$\supseteq$''.
By Prop. \ref{hskeletal2}, $f^*[T]\subseteq f_{-1}[T]$ and so, for $x=\bim T\in T$,
$$\clr(f_{-1}[T])\supseteq \cf (f^*(x))=f_{-1}[\cf(x)]=f_{-1}[\clr T].\qedhere$$
\end{proof}

\begin{remas} (1) Contrarily to what happens with  topological spaces, where  (a)$\Leftrightarrow$(c) holds (\cite[Exercise 1.4.C]{RE}), the implication (c)$\Rightarrow$(a) is not true in locales (Johnstone \cite{PTJ}, answering a question posed in \cite{CGT}). A  counterexample is  the embedding $\mathcal{B}(M)\hookrightarrow M$ of the Booleanization $\mathcal{B}(M)=\{b\in M\mid b^{**}=b\}$, which satisfies (c), since all sublocales of a Boolean algebra are closed, but is rarely open.

Localic maps with property (c) are the
{\em hereditary skeletal maps} (Johnstone \cite{PTJ}). Johnstone characterized them as the $f\colon L\to M$ such that
\[
f^*(b\to c)\to f^*(c)=(f^*(b)\to f^*(c))\to f^*(c)
\] for every $b,c\in M$ (\cite[Lemma 4.1]{PTJ}), and proved that an hereditarily skeletal map $f\colon L\to M$ is open whenever $f[L]$ is a complemented sublocale of $M$.

\smallskip
\nid
(2) For a general categorical treatment of this question see
\cite{CGT,GT}.

\smallskip
\nid
(3)
The case $T=\of(b)$ in (c) amounts to $\cf(f^*(b^*))=\cf(f^*(b)^*)$, that is, to $f^*(b^*)=f^*(b)^*$. Frame homomorphisms satisfying this condition are called {\em nearly open} in \cite{BP}  and studied in detail there (see also Johnstone \cite{PTJ2,PTJ}).

There are further important weak variants of open localic maps, namely, the {\em skeletal maps} \cite{PTJ} ({\em weakly open} in \cite{BP}), defined by the identity $f^*(b^*)^*=f^*(b)^{**}$, and the {\em
sub-open maps} \cite{PTJ3,PTJ} that preserve the Heyting implication:

\bigskip
$$
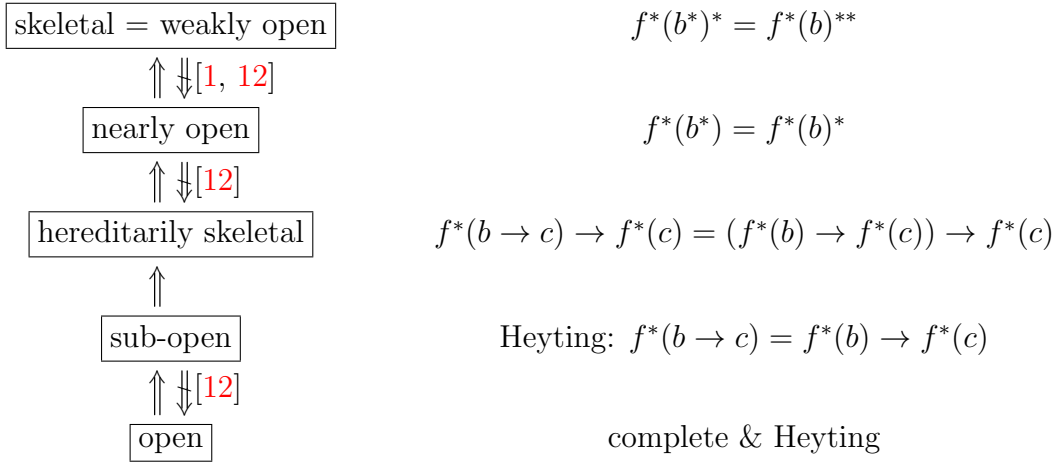
\xymatrix@C=30pt@R=15pt{\fbox{skeletal = weakly open} \ar@<1ex>@{=>}[d]^{\mbox{\cite{BP,PTJ}}} |{\SelectTips{cm}{}\object@{/}}|{} & f^*(b^*)^*=f^*(b)^{**}\\
\fbox{nearly open} \ar@<1ex>@{=>}[u]^{}   \ar@<1ex>@{=>}[d]^{\mbox{\cite{PTJ}}} |{\SelectTips{cm}{}\object@{/}}|{} & f^*(b^*)=f^*(b)^*\\
\fbox{hereditarily skeletal} \ar@<1ex>@{=>}[u]^{}  & f^*(b\to c)\to f^*(c)=(f^*(b)\to f^*(c))\to f^*(c)\\
\fbox{sub-open} \ar@<1ex>@{=>}[u]^{} \ar@<1ex>@{=>}[d]^{\mbox{\cite{PTJ}}} |{\SelectTips{cm}{}\object@{/}}|{} & \mbox{Heyting: }f^*(b\to c)=f^*(b)\to f^*(c)\\
\fbox{open} \ar@<1ex>@{=>}[u]^{} &  \mbox{complete \& Heyting} } $$

\bigskip
An hereditarily skeletal map is precisely a localic map $f\colon L\to M$ such that its pullback along any sublocale embedding $T\hookrightarrow M$ is skeletal. It is still an open problem whether there exists a dividing example for the classes of hereditarily skeletal maps and sub-open maps.

The class of open maps coincides with the classes of respectively {\em stably sub-open maps} (the ones whose pullbacks along any localic map are still sub-open), {\em stably nearly open maps} (all pullbacks of $f$ are nearly open), and
{\em stably skeletal maps} (all pullbacks of $f$
are skeletal) \cite[Thm. 4.7]{PTJ}.

\smallskip
\nid
(4) Let $f\colon L\to M$ be a localic map.  For a sublocale $T$ of $M$,
\begin{align*}
f_{-1}[\intr T]&=f_{-1}[\of(\tbigvee \{b\mid \of(b)\subseteq T\})]\\&=\of(f^*(\tbigvee \{b\mid \of(b)\subseteq T\}))=\of(\tbigvee \{f^*(b)\mid \of(b)\subseteq T\}).
\end{align*}
On the other hand,
\begin{align*}
\intr(f_{-1}[T])&=\tbigvee \{\of(a)\mid \of(a)\subseteq f_{-1}[T]\}\\&=\tbigvee \{\of(a)\mid f[\of(a)]\subseteq T\}=\of(\tbigvee \{a\mid f[\of(a)]\subseteq T\}).
\end{align*}
Hence, $f$ is open if and only if $$\tbigvee \{f^*(b)\mid \of(b)\subseteq T\}=\tbigvee \{a\mid f[\of(a)]\subseteq T\}.$$
\end{remas}

\begin{lemm}
Let $f\colon L\to M$ be a meet-preserving map between locales. Then
 $$\intr(f^{-1}[T])=\intr(f_{-1}[T])\quad\mbox{for every }T\in{\MC}(M).$$
\end{lemm}
\begin{proof}
     It suffices to check the inclusion ``$\subseteq$''.
    Since $f$ is a meet-preserving map, $f^{-1}[T]$ is a meet-subset and
    $\intr(f^{-1}[T])\subseteq f^{-1}[T].$
    On the other hand, since $\intr(f^{-1}[T])$ is a sublocale, the former condition implies that $\intr(f^{-1}[T])\subseteq f_{-1}[T]$ and we may conclude that
    $\intr(f^{-1}[T])\subseteq \intr( f_{-1}[T]).$
\end{proof}

\begin{prop}
A plain map $f\colon L\to M$ is an open localic map if and only if
$$(\intr(f^{-1}[T]))^\co=f^{-1}[(\intr T)^\co]\quad \mbox{ for all }T\in\SL(M).$$
\end{prop}

\begin{proof}
Let $f\colon L\to M$ be an open localic map.
By the preceding lemma and by Cor. \ref{propfopen}, we conclude that $f_{-1}[\intr T]=\intr f_{-1}[T]=\intr f^{-1}[T]$.
Since $f$ is localic and $(\intr T)\co$ is a closed sublocale it follows that
$$f^{-1} [(\intr T)\co]=f_{-1} [(\intr T)\co]=f_{-1} [\intr T]\co=(\intr f^{-1} [T])\co.$$

Conversely, applying the hypothesis to $T=\of (b)$ we get $$f^{-1}[\cf (b)]=(\intr (f^{-1}[\of (b)]))\co.$$
Then $f$ is a localic map, by Prop. \ref{theofloccaract}.

Moreover, for any sublocale $T$ of $M$, we have, by the previous lemma,
$$
\intr(f_{-1}[T])=(\intr(f^{-1}[T])\co)\co =f^{-1}[(\intr T)\co]\co
    =f_{-1}[(\intr T)\co]\co
    =f_{-1}[\intr T].
$$
Hence $f$ is also an open map (by Cor. \ref{propfopen}).
\end{proof}

Now we extend
Prop. \ref{hskeletal2} from open maps to hereditary skeletal maps:

\begin{prop}\label{hskeletal1}
The following are equivalent for a localic map $f\colon L\to M$:
\begin{enumerate}[\em (i)]
\item $f$ is hereditarily skeletal, that is, $\clr(f_{-1}[T])=f_{-1}[\clr T]$ for every $T\in\SL(M)$.
\item $\tbigwedge f_{-1}[T]\le f^*(\tbigwedge T)$  for every $T\in\SL(M)$.
\item $f^*(\tbigwedge T) \in f_{-1}[T]$  for every $T\in\SL(M)$.
\end{enumerate}
\end{prop}

\begin{proof} (i)$\Leftrightarrow$(ii):
Since \[f_{-1}[\clr T]=f_{-1}[\tbigcap\{\cf(b)\mid \cf(b)\supseteq T\}]=\tbigcap\{\cf(f^*(b))\mid \cf(b)\supseteq T\}\] and
$\clr(f_{-1}[T])=\tbigcap\{\cf(a)\mid \cf(a)\supseteq f_{-1}[T]\}$, the converse containment of (C2) amounts to
$$f_{-1}[T]\subseteq \cf(a)\ \Rightarrow \ \tbigcap\{\cf(f^*(b))\mid \cf(b)\supseteq T\}\subseteq \cf(a)$$
that is,
$$a\le \tbigwedge f_{-1}[T] \ \Rightarrow \ a\le \tbigvee\{f^*(b)\mid b\le \tbigwedge T\}$$
or, in other words,
\begin{equation*}
\tbigwedge f_{-1}[T]\le f^*(\tbigwedge T).
\end{equation*}

\smallskip
\nid
(ii)$\Rightarrow$(iii):
The inequality $f^*(\tbigwedge T)\le \tbigwedge f_{-1}[T]$ is always true (indeed, for
$x=\tbigwedge f_{-1}[T]\in f_{-1}[T]$, $f(x)\in T$ so $f^*(\tbigwedge T)\le f^*f(x)\le x$). Hence $f^*(\tbigwedge T)=x\in f_{-1}[T]$.

\smallskip
\nid
(iii)$\Rightarrow$(ii): Obvious.
\end{proof}

\begin{rema}Consider the functor $\TL\colon {\bf Loc}\to{\bf Loc}$ defined by \[\TL(L)=\SL(L)\op\quad\mbox{ and }\quad \TL(f)=f[-]\] and the {\em dissolving maps} \cite{JI} $\gamma_L\colon \TL(L)\to L$ defined by $\gamma_L(S)=\bim S$ (see \cite{PP23} for more information on the dissolution of a locale). The dissolving maps are localic maps (in fact, they are the right adjoints of the natural frame embeddings $\cf_L= (a\mapsto \cf(a))\colon L\to\TL(L)$) and establish a natural transformation $\gamma\colon \TL\stackrel{\cdot}{\to} \mathrm{Id}$ (\cite[1.4]{JI}). Therefore, for each $f\colon L\to M$ in {\bf Loc}, the following square commutes
\[
\xymatrix@C=45pt@R=45pt{
\TL(L) \ar[r]^{\gamma_L} \ar[d]_{f[-]} & L \ar[d]^{f} \\
\TL(M)   \ar[r]_{\gamma_M} &  M }
\]
as well as the square of their left adjoints (dotted arrows in the diagram below)
\[
\xymatrix@C=45pt@R=45pt{
\TL(L) \ar@/_/[r]_{\gamma_L}^{\bot} \ar@/^/[d]^{f[-]} & L \ar@{..>}@/_/[l]_{\cf_L} \ar@/^/[d]^{f} \\
\TL(M) \ar@{..>}@/^/[u]^{f_{-1}[-]}_{\dashv}   \ar@/_/[r]_{\gamma_M}^{\bot} &  M \ar@{..>}@/_/[l]_{\cf_M} \ar@{..>}@/^/[u]^{f^*}_{\dashv} }
\]

In the last diagram, one has always the inequality \[\gamma_L \circ f_{-1}[-] \ge f^*\circ \gamma_M\] (since $\gamma_L\circ \cf_L\ge 1$ and $\cf_M\circ \gamma_M\le 1$).
Prop. \ref{hskeletal1} shows that $f$ is hereditarily skeletal precisely when one has the other inequality, hence the equality
\begin{equation}\label{r571}\tag{$*$}
\gamma_L \circ f_{-1}[-] = f^*\circ \gamma_M.
\end{equation}

When $f$ is open, we have moreover the adjunction $f_{!}\dashv f^*$ and then \eqref{r571} implies
$$f_{!}\circ \gamma_L\circ f_{-1}[-]\le \gamma_M$$ (that is, $f_{!}(\bim f_{-1}[T])\le \bim T$ for every $T\in\SL(M)$).
\end{rema}

\medskip
\section{Commutativity of images with closure and interior}

In this final section, we collect some results about the the corresponding situation for images (instead of preimages).
Not surprisingly (as for spaces \cite{RE}), this case is not so rich.

Similarly as open maps, we define {\em closed maps} as (plain) maps $f\colon L\to M$ between locales
such that the image $f[\cf(a)]$ of every closed sublocale $\cf(a)$ is closed. Closed meet-preserving maps are easy to characterize (in fact, the proof in \cite[Prop. III.7.3]{PP12} for localic maps also holds for arbitrary meet-preserving maps); they are precisely the meet-preserving maps $f\colon L\to M$ that satisfy each one of the following equivalent conditions:

\begin{enumerate}[(i)]
\item For every $a \in L$, $f[\cf(a)] = \cf(f(a))$.
\item For every $a \in L$ and $b \in  M$, $f(a \vee f^*(b)) = f(a) \vee b$.
\item For every $a \in L$ and $b, c \in M$, $c \le f(a) \vee b$ iff $f^*(c) \le a \vee f^*(b)$.
\item For every $a \in L$ and $b, c \in M$, $f(a)\vee b = f(a) \vee c$ iff $a \vee f^*(b) = a \vee  f^*(c)$.
\end{enumerate}

\begin{lemm}
Let $f\colon L\to M$ be a plain map between locales. Then  $f[\clr X]\subseteq \clr f[X]$ for all $X\subseteq L$ iff $f$ is meet-preserving.
\end{lemm}

\begin{proof}   ``$\Rightarrow$'': By Cor. \ref{corogalois}, it suffices to check that $\clr(f^{-1}[\cf (b)]) \subseteq f^{-1}[\cf (b)]$ for every $b \in M$. This follows immediately from  \[f[\clr f^{-1}[\cf(b)]]\subseteq \clr f[ f^{-1}[\cf (b)]] \subseteq \cf (b).\]
``$\Leftarrow$'': Since $f$ preserves meets,
$\bim f[\clr X]=f(\bim \clr X)=f(\bim X)=\bim f[X]$.
\end{proof}

\begin{lemm}
Let $f\colon L\to M$ be a monotone map between locales. Then  $f[\clr S]\subseteq \clr f[S]$ for all $S\in \SL(L)$ iff $f$ preserves meets of sublocales.
\end{lemm}
\begin{proof}
    Let $S\in \SL(L)$. Since $f$ is monotone we have $f(\bim S)\leq \bim f[S]$. On the other hand,
    from $f[\cf(\bim S)]=f[\clr S]\subseteq \clr f[S]=\cf(\bim f[S])$ it follows that $f(\bim S)\ge \bim f[S]$. Conversely, for any $a\ge \bim S$, $f(a)\ge f(\bim S)=\bim f[S]$.
\end{proof}

\begin{prop}
The following are equivalent for a  plain map $f\colon L\to M$ between locales:
\begin{enumerate}[\em (i)]
\item $\clr f[S] \subseteq f[\clr S]$ for all $S\in\SL(L)$.
\item $\clr f[X]\subseteq f[\clr X]$ for all $X\subseteq L$.
\item $f$ is closed.
\end{enumerate}
\end{prop}

\begin{proof}
(i)$\Rightarrow$(ii): For every $X\subseteq L$, $\clr X \in\SL(L)$ hence $\clr f[X]\subseteq \clr f[\clr X]\subseteq f[\clr(\clr X)]=f[\clr X]$.

\smallskip
\nid
(ii)$\Rightarrow$(iii): For every $a\in L$,  $\clr f[\cf (a)] \subseteq f[\cf (a)]$ hence $f[\cf(a)]$ is closed.

\smallskip
\nid
(iii)$\Rightarrow$(i): If $f$ is a closed map then, for any $S\in\SL(L)$, $f[\clr S]$ is closed hence $\clr f[S]\subseteq \clr f[\clr S]=f[\clr S]$.
\end{proof}

Then, immediately, we have:

\begin{coro}
Let $f\colon L\to M$ be a plain map between locales. Then $f[\clr X]=\clr f[X]$ for all $X\subseteq L$ if and only if $f$ is closed and meet-preserving.
\end{coro}

\begin{coro}
The following are equivalent for a  meet-preserving map $f\colon L\to M$ between locales:
\begin{enumerate}[\em (i)]
\item $\clr f[S] = f[\clr S]$ for all $S\in\SL(L)$.
\item $\clr f[X] = f[\clr X]$ for all $X\subseteq L$.
\item $f$ is closed.
\end{enumerate}
\end{coro}

Finally, for the interior operator we have:

\begin{prop}
    Let $f\colon L\to M$ be a meet-preserving map between locales. Then $f$ is open if and only if $f[\intr S] \subseteq \intr f[S]$ for every $S\in\SL(L)$.
\end{prop}

\begin{proof}
    If $f$ is an open map then $f[\intr S]=\intr f[\intr S]\subseteq \intr f[S]$. Conversely, for any $S=\of (a)$, $f[\of (a)]\subseteq \intr f[\of(a)]$, and since $f[\of (a)]$ is a meet-subset it follows that $f[\of (a)]=\intr f[\of(a)]$.
\end{proof}

\section*{Acknowledgements} The work of the second named author was partially supported by the Centre for Mathematics of the University of Coimbra
(UIDB/00324/2020, funded by the Portuguese Government through FCT/MCTES).


\bigskip
\begin{thebibliography}{99}

\bibitem{BP} B. Banaschewski and A. Pultr, Variants of openness, {\em Appl. Categ. Structures} 2 (1994) 331--350.


\bibitem{CGT}
M.M. Clementino, E. Giuli and W. Tholen, A functional approach to general topology, in: {\em Categorical Foundations},
Encyclopedia Math. Appl., vol. 97, Cambridge Univ. Press, 2004, pp. 103--163.


\bibitem{RE}
R. Engelking, {\em General Topology}, revised and completed edition, Heldermann, Berlin,
1989.

\bibitem{ME} M. Erné, Adjunctions and Galois connections: origins, history and development, in: {\em Galois Connections and Applications}, Mathematics and its Applications, vol. 565, Springer, 2004, pp. 1--138.

\bibitem{EKMS} M. Erné, J. Koslowski, A. Melton and G.E. Strecker, A primer on Galois connections, in:
{\em Papers on general topology and applications} (Madison, WI, 1991), Ann. New York Acad. Sci.,
vol. 704, 1993, pp. 103--125.


\bibitem{EPP22} M. Erné, J. Picado and A. Pultr,
Adjoint maps between implicative semilattices and continuity of localic maps,
{\em Algebra Univers.} 83 (2022) article n. 13.

\bibitem{FPP}  M.J. Ferreira, J. Picado and S. Pinto,  Remainders in pointfree topology
     {\em Topology Appl.} 245 (2018) 21--45.

\bibitem{GT} E. Giuli and W. Tholen, Openness with respect to a closure operator, {\em Appl. Categ. Structures} 8 (2000) 487--502.


\bibitem{JI} J. Isbell, Atomless parts of spaces, {\em Math. Scand.} 31 (1972) 5--32.

\bibitem{PTJ3} P.T. Johnstone, Open maps of toposes, {\em Manuscripta Math.} 31 (1980) 217--247.

\bibitem{PTJ2} P.T. Johnstone, Factorization theorems for geometric morphisms, II, in: {\em Categorical Aspects of Topology
and Analysis}, Lecture Notes in Math., vol. 915, Springer, 1982, pp. 216--233.

\bibitem{PTJ} P.T. Johnstone, Complemented sublocales and open maps, {\em Ann. Pure Appl. Logic} 137 (2006) 240--255.

\bibitem{JT84} A. Joyal and M. Tierney,
{\em An Extension of the Galois Theory of Grothendieck},
Memoirs of the American Mathematical Society, vol. 309. AMS, Providence, 1984.


\bibitem{EL06} E.S. Letzer, On continuous and adjoint morphisms between non-commutative spectra, {\em Proc. Edinburgh Math. Soc.} 49 (2006) 367--381.

\bibitem{EL15} E.S. Letzter, Continuity is an adjoint functor,
{\em The Amer. Math. Monthly} 122 (2015) 70--74.

\bibitem{SM} Saunders MacLane, {\em Categories for the Working Mathematician}, Graduate Texts in Mathematics 5,
Springer, 1971 (2nd ed. 1997).

\bibitem{PP08} J. Picado and A. Pultr, {\em  Locales treated mostly in a covariant way},
     Textos de Matemática, DMUC, vol. 41, 2008.

\bibitem{PP12} J. Picado and A. Pultr, {\em Frames and locales. Topology without points}, Frontiers in Mathematics, Springer Basel, 2012.

\bibitem{PP21} J. Picado and A. Pultr, {\em Separation in point-free topology}, Birkhäuser/Springer Cham, 2021.


\bibitem{PP22} J. Picado and A. Pultr, Notes on point-free topology,
 in: {\em New Perspectives in Algebra, Topology and Categories} (ed. by M.M. Clementino, A. Facchini, M. Gran),
     Coimbra Mathematical Texts, vol. 1, Springer, 2022, pp. 173--223.

\bibitem{PP23} J. Picado and A. Pultr, {\em Notes on sublocales and dissolution}, DMUC Preprint 23-26, 2023 (submitted for publication).


\bibitem{PPT} J. Picado, A. Pultr and A. Tozzi,
Ideals in Heyting semilattices and open homomorphisms,
{\em Quaestiones Math.} 30 (2007) 391--405.

\end{thebibliography}
\end{document}